\PassOptionsToPackage{dvipsnames,svgnames}{xcolor}
\documentclass[ecta,nameyear]{econsocart}

\RequirePackage[colorlinks,citecolor=blue,linkcolor=blue,urlcolor=blue]{hyperref}

\startlocaldefs

\usepackage{amsmath, amssymb, amsthm, amsfonts}
\usepackage{mathtools}
\providecommand{\coloneq}{\coloneqq}
\usepackage[mathscr]{euscript}
\usepackage{graphicx}
\usepackage{tikz}
\usepackage{verbatim}
\usepackage{xcolor}
\usepackage{mathrsfs}
\usepackage{bbm}
\usepackage{bm}

\usepackage{enumitem}
\setlist[enumerate]{itemsep=2pt,topsep=3pt}
\setlist[itemize]{itemsep=2pt,topsep=3pt}
\setlist[enumerate,1]{label={\upshape (\roman*)}}

\renewcommand{\leq}{\leqslant}
\renewcommand{\geq}{\geqslant}

\providecommand{\inner}[1]{\left\langle{#1}\right\rangle}

\newcommand{\setntn}[2]{ \{ #1 : #2 \} }

\newcommand{\preqsd}{\preceq}

\newcommand{\1}{\mathbbm 1}

\newcommand{\given}{\, | \,}

\newcommand*\diff{\mathop{}\!\mathrm{d}}

\let\emptyset\varnothing

\newcommand{\bB}{\mathcal B}

\newcommand{\mM}{\mathcal M}

\newcommand{\pP}{\mathcal P}

\newcommand{\RR}{\mathbbm R}

\newcommand{\NN}{\mathbbm N}

\newcommand{\PP}{\mathbbm P}

\newcommand{\EE}{\mathbbm E}

\newcommand{\Xsf}{\mathsf X}

\renewcommand{\phi}{\varphi}
\renewcommand{\epsilon}{\varepsilon}

\usepackage[capitalize,nameinlink]{cleveref}
\crefname{figure}{Figure}{Figures}
\usepackage{thmtools}

\declaretheorem[style=plain, name=Theorem]{theorem}

\declaretheorem[style=plain, name=Lemma, sibling=theorem]{lemma}
\declaretheorem[style=plain, name=Proposition, sibling=theorem]{proposition}

\newcommand{\navy}[1]{\textcolor{MidnightBlue}{\emph{#1}}}

\endlocaldefs

\begin{document}

\begin{frontmatter}

\title{Monotone Mixing and Distribution Dynamics}
\runtitle{Monotone Mixing and Distribution Dynamics}

\begin{aug}
\author[add1]{\fnms{Takashi}~\snm{Kamihigashi}\ead[label=e1]{tkamihig@rieb.kobe-u.ac.jp}}
\author[add2]{\fnms{Qingyin}~\snm{Ma}\ead[label=e2]{qingyin.ma@outlook.com}}
\author[add3]{\fnms{John}~\snm{Stachurski}\ead[label=e3]{j-stachurski@grips.ac.jp}}

\address[add1]{%
\orgdiv{Center for Computational Social Science},
\orgname{Kobe University}}

\address[add2]{%
\orgdiv{International School of Economics and Management},
\orgname{Capital University of Economics and Business}}

\address[add3]{%
\orgname{National Graduate Institute for Policy Studies}}
\end{aug}

\begin{abstract}
    Many economic applications establish stability of Markov dynamics using the
    monotone mixing condition (MMC) of Hopenhayn and Prescott (1992). Working in
    the same setting as that paper, we introduce a weak monotone mixing
    condition (WMMC), phrased in terms of the reversal of rank between
    populations, and show that it is strictly weaker than the MMC and both
    necessary and sufficient for global stability under the Kolmogorov metric.
     We apply these results to a
    small open economy version of the Aiyagari model, showing how the WMMC can
    be used to establish global stability in a setting where the MMC fails. In
    addition, we show that, when the state space is one-dimensional, the MMC and
    WMMC coincide, implying that the original MMC is necessary as well as
    sufficient in this setting.
\end{abstract}



\end{frontmatter}

\section{Introduction}

Stochastic stability plays a major role in economic theory. 
\cite{hopenhayn1992stochastic} contributed to this theory
by analyzing stability in a Markov setting where transition dynamics are monotone.
They showed that, in their
setting, global stability is implied by a simple mixing condition.  This
monotone mixing condition (MMC) has become a standard approach to establishing
stability properties over a large range of applications, including international
trade, human capital, business cycles, inequality, intergenerational mobility,
and labor markets.\footnote{See, e.g., \cite{marcet2007incomplete},
\cite{antunes2007startup}, \cite{morand2007stationary},
\cite{hidalgo2009endogenous}, \cite{samaniego2008technical},
\cite{le2022managing}, \cite{light2022mean}, \cite{balbus2025markov}, and
\cite{kam2025inflation}.  In the context of wealth and income dynamics,
\cite{riosrull1998computing} calls the MMC ``the American Dream and the
American Nightmare'' condition, since it requires that the poorest agents have
positive probability of becoming rich and the richest have positive probability
of becoming poor.}

In this note we extend \cite{hopenhayn1992stochastic}. Working in the same environment, we
introduce a weak monotone mixing condition (WMMC), stated in terms of 
partial rank reversal between populations evolving from different initial
conditions. We
then show that this weaker condition is both necessary and sufficient for global
stability, and that it delivers an explicit
geometric rate of convergence. In addition, we show that the WMMC
is strictly weaker than the MMC, and that this matters for economic applications
with multiple state variables.  As a further contribution, we show that the
WMMC and the MMC agree when the state is one-dimensional.  This implies that the
original MMC is in fact necessary as well as sufficient in such settings.

In terms of techniques, one factor driving our results is that we use the
Kolmogorov metric to define global stability, rather than the notion of weak
convergence applied by \cite{hopenhayn1992stochastic}.  The key benefit of the
Kolmogorov metric is that it respects the underlying order on the state space.
Convergence in the Kolmogorov metric is stronger than weak convergence in the setting of
\cite{hopenhayn1992stochastic}, which assists our necessity results and makes our
sufficiency results stronger. The Kolmogorov metric is often used by
econometricians, since it underpins Kolmogorov--Smirnov tests.

\cite{hopenhayn1992stochastic} built on earlier work by
\cite{razin1979stochastic}, \cite{bhattacharya1988asymptotics}, and
\cite{stokeyec}. Numerous authors have since considered global stability in
monotone environments similar to those studied by
\cite{hopenhayn1992stochastic}, often with the goal of weakening assumptions on
the state process or the state space.  (See, for example,
\cite{kamihigashi2014stochastic}, \cite{foss2018stochastic},
\cite{kamihigashi2019unified}, \cite{light2022mean}, \cite{light2026invariant},
or \cite{foss2026compressibility}.)  To the best of our knowledge, this paper is the
first to obtain an exact characterization of stability (necessity and sufficiency) 
in terms of monotone mixing.

\section{Setup}\label{s:setup}

In this section we introduce the key components of our analysis.
Our setting is that of \cite{hopenhayn1992stochastic}.
In particular, the state space $\Xsf$ is a compact
metric space with Borel sets $\bB$, endowed with a closed partial order $\preceq$,
 and $\Xsf$ has both a least element $a$ and a greatest element $b$.

Given $c,d \in \Xsf$ we write $[c,d]$ for the set of $x\in \Xsf$ with $c \preceq x
\preceq d$.
We write $b\Xsf$ for the bounded Borel measurable functions from $\Xsf$ to
$\RR$, $bc\Xsf$ for the continuous functions in $b\Xsf$, and $\pP$ for the Borel
probability measures on $\Xsf$, calling elements of $\pP$ \navy{distributions}.
A function $h \in b\Xsf$ is called \navy{increasing} if $x \preceq x'$ implies $h(x)
\leq h(x')$, and $B \in \bB$ is called \navy{increasing} if $\1_B$ is increasing.
We let $ib\Xsf$ denote the increasing functions in $b\Xsf$ and $i\bB$ the
increasing sets in $\bB$.  For $\mu \in \pP$ and $h \in b\Xsf$ we set
$\inner{\mu, h} \coloneq \int h \diff \mu$. Also, $\delta_x$ is the
probability measure concentrated at $x$.

Let $\mM$ denote the finite Borel measures on $\Xsf$.  Given $\mu, \nu \in
\mM$, we write $\mu \preqsd \nu$, and say that $\mu$ is \navy{stochastically
dominated} by $\nu$, if $\mu(\Xsf) = \nu(\Xsf)$ and $\mu(I) \leq \nu(I)$ for
all $I \in i\bB$.  On $\pP$ this is the standard notion, equivalent to
$\inner{\mu, h} \leq \inner{\nu, h}$ for all $h \in ib\Xsf$.
Moreover, $\preqsd$ is a partial order on
$\pP$ \citep[Theorem~2]{kamae1978stochastic}, and $x \preceq x'$ implies
$\delta_x \preqsd \delta_{x'}$.  We also use the pointwise order on $\mM$,
writing $\mu' \leq \mu$ when $\mu'(B) \leq \mu(B)$ for all $B \in \bB$ and
calling $\mu'$ a \navy{subpopulation} of $\mu$.

A \navy{stochastic kernel} on $(\Xsf, \bB)$ is a map $P \colon \Xsf \times \bB
\to [0,1]$ such that $x \mapsto P(x, B)$ is Borel measurable for each $B \in
\bB$, while $B \mapsto P(x, B)$ is a probability measure for each $x \in \Xsf$.
We write $P_x \coloneq P(x, \cdot)$ and set
\begin{equation*}
    (\mu P)(B) \coloneq \int P(x, B) \, \mu(\diff x)
    \quad \text{and} \quad
    (Ph)(x) \coloneq \int h(x') \, P(x, \diff x') ,
\end{equation*}
so that $\mu \mapsto \mu P$ maps $\pP$ into itself while $h \mapsto Ph$ maps
$b\Xsf$ into itself.  We write $P^t$ for the $t$-th iterate of $P$, with $P^0_x
\coloneq \delta_x$.
A kernel $P$ is called \navy{increasing} if $\mu \preqsd \nu$ implies $\mu P
\preqsd \nu P$ or, equivalently, if $P_x \preqsd P_{x'}$ whenever $x \preceq
x'$ \citep[Proposition~1]{kamae1977stochastic}.  Evidently
$P^t$ is increasing for every $t$ whenever $P$ is.

A distribution $\mu^* \in \pP$ is called \navy{stationary} for $P$ if $\mu^* =
\mu^* P$.  We call $P$ \navy{globally stable} if $P$ has a unique stationary
distribution $\mu^* \in \pP$ and, in addition, $\kappa(\mu P^t, \mu^*) \to 0$ as
$t \to \infty$ for every $\mu \in \pP$, where $\kappa$ is the 
\navy{Kolmogorov metric}
\begin{equation}\label{eq:kol}
    \kappa(\mu, \nu)
    \coloneq 2 \sup_{I \in i\bB} |\mu(I) - \nu(I)|
    \qquad (\mu, \nu \in \pP) .
\end{equation}
Since $\Xsf$ is compact, convergence in $\kappa$ implies the weak convergence
used in \cite{hopenhayn1992stochastic}; see, e.g., \cite{kamihigashi2019unified}.

\section{Main Result}\label{s:main}

Throughout this section and the next, $P$ is an increasing stochastic kernel
on $(\Xsf, \bB)$.  We introduce a weak monotone mixing condition and show
that it characterizes global stability.  Major proofs are deferred to
\cref{s:proofs}.

To begin, fix $\mu$ and $\nu$ in $\pP$. For the purpose of this discussion, it will be helpful to regard $\mu$ and $\nu$
as two populations of equal size, each distributed over the state space
$\Xsf$. A point of $\Xsf$ is a possibly multidimensional list of attributes of
an individual and $\preceq$ ranks attribute bundles.  We say that $\mu$ is
\navy{partially dominated} by $\nu$ if there exist $\epsilon > 0$ and $\mu'$,
$\nu'$ in $\mM$ with
\begin{equation}\label{eq:pd}
    \mu' \leq \mu,
    \qquad
    \nu' \leq \nu,
    \qquad
    \mu'(\Xsf) = \nu'(\Xsf) = \epsilon
    \qquad \text{and} \qquad
    \mu' \preqsd \nu' .
\end{equation}
In words, a fraction $\epsilon$ of the population $\mu$ can be matched with an
equally large fraction of $\nu$ in such a way that the first ranks weakly below
the second.  When \eqref{eq:pd} holds for a given $\epsilon$, we say that $\mu$
is partially dominated by $\nu$ \navy{at level} $\epsilon$.\footnote{A pair
    $(\mu', \nu')$ satisfying \eqref{eq:pd} is called an ordered component pair
    in \cite{kamihigashi2019unified}. The partial dominance terminology comes from
\cite{kamihigashi2020partial}.}

\cref{f:pd} illustrates the definition when $\Xsf \subset \RR$ and $\mu$ and
$\nu$ have continuous CDFs $F_\mu$ and $F_\nu$.  Here $F_\mu \leq F_\nu$, so
$\nu \preqsd \mu$.  Nonetheless, $\mu$ is partially dominated by $\nu$ at level
$\epsilon$.  To see this, choose $x_\mu$ and $x_\nu$ with $F_\mu(x_\mu) =
\epsilon$ and $F_\nu(x_\nu) = 1 - \epsilon$, and let $\mu'$ and $\nu'$ be the
restrictions of $\mu$ to $[a, x_\mu]$ and of $\nu$ to $[x_\nu, b]$.  Each has
mass $\epsilon$.  In the figure $x_\mu \leq x_\nu$, so all of the mass of $\mu'$ lies
weakly below all of the mass of $\nu'$, and hence $\mu' \preqsd \nu'$.  

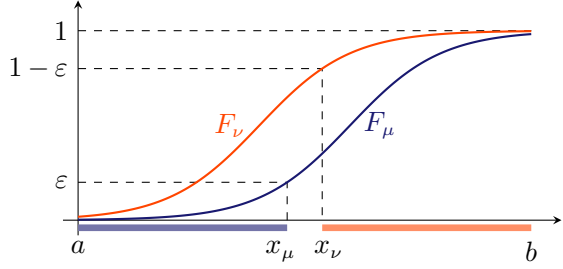
\begin{figure}[t]
    \centering
    \begin{tikzpicture}[>=stealth]
        \draw[->] (-0.2,0) -- (6.4,0);
        \draw[->] (0,-0.2) -- (0,2.9);
        \draw[dashed] (0,2.5) -- (6,2.5);
        \node[left] at (0,2.5) {$1$};
        \draw[dashed] (0,0.5) -- (2.768,0.5) -- (2.768,0);
        \draw[dashed] (0,2.0) -- (3.232,2.0) -- (3.232,0);
        \node[left] at (0,0.5) {$\epsilon$};
        \node[left] at (0,2.0) {$1-\epsilon$};
        \draw[line width=2.5pt, MidnightBlue!60] (0,-0.1) -- (2.768,-0.1);
        \draw[line width=2.5pt, OrangeRed!60] (3.232,-0.1) -- (6,-0.1);
        \draw[thick, OrangeRed]
            plot[domain=0:6, samples=120] (\x, {2.5/(1+exp(-(\x-2.4)/0.6))});
        \draw[thick, MidnightBlue]
            plot[domain=0:6, samples=120] (\x, {2.5/(1+exp(-(\x-3.6)/0.6))});
        \node[OrangeRed, left] at (2.35,1.25) {$F_\nu$};
        \node[MidnightBlue, right] at (3.65,1.25) {$F_\mu$};
        \node[below] at (0,-0.15) {$a$};
        \node[below] at (6,-0.15) {$b$};
        \node[below] at (2.68,-0.15) {$x_\mu$};
        \node[below] at (3.32,-0.15) {$x_\nu$};
    \end{tikzpicture}
    \caption{Partial dominance of $\mu$ by $\nu$ at level $\epsilon$ when
    $\Xsf \subset \RR$.  The lowest fraction $\epsilon$ of $\mu$, supported
    on $[a, x_\mu]$, ranks below the highest fraction $\epsilon$ of $\nu$,
    supported on $[x_\nu, b]$.}
    \label{f:pd}
\end{figure}

We say that $P$ satisfies the \navy{weak monotone mixing condition} (WMMC) if
there exists an $n \in \NN$ such that $P^n_b$ is partially dominated by $P^n_a$.
To interpret this definition, observe that, since $a \preceq b$ and $P$ is
increasing, we have $P^n_a \preqsd P^n_b$ at every $n$. This means that the
descendants of the ``high class'' initial condition $b$ stochastically dominate
the descendants of the ``low class'' initial condition $a$ at every horizon.  At
the same time, this statement is aggregate, at the level of
distributions, and does not rule out reversal at the level of
individuals.  The WMMC says that such a reversal occurs at some horizon, for a
positive fraction of each population.

We remark that, for the WMMC to hold, the reversal need only be weak:  it suffices that
the low-class descendants catch up with the high-class descendants, and no part of
either population need strictly overtake the other. The name WMMC reflects the
fact that the condition is strictly weaker than the monotone mixing condition of
\cite{hopenhayn1992stochastic}. See \cref{s:mmc} for details.

We can now state our main theorem.

\begin{theorem}\label{t:main}
    Let $P$ be an increasing stochastic kernel on $(\Xsf, \bB)$.  Then $P$ is
    globally stable if and only if $P$ satisfies the weak monotone mixing
    condition.  In this case, if $n \in \NN$ and $\epsilon > 0$ are such that
    $P^n_b$ is partially dominated by $P^n_a$ at level $\epsilon$, then the
    stationary distribution $\mu^*$ obeys
    \begin{equation}\label{eq:rate}
        \kappa(\mu P^t, \mu^*)
        \leq 4 (1 - \epsilon)^{\lfloor t/n \rfloor}
        \qquad \text{ for every $\mu \in \pP$ and every $t \geq 0$}.
    \end{equation}
\end{theorem}

In \eqref{eq:rate},
$\epsilon$ is a fraction of the two groups whose ranking can be reversed over
$n$ periods, and it is this fraction that governs how quickly the distribution converges.

\section{Relationship to the MMC}\label{s:mmc}

In this section we relate the WMMC to the monotone mixing condition of
\cite{hopenhayn1992stochastic}, and show that the two coincide when the state
space is one-dimensional.

\subsection{Straddle Points}\label{ss:straddle}

To connect the WMMC with the mixing condition of
\cite{hopenhayn1992stochastic}, it is helpful to introduce the following
concept: $s \in \Xsf$ is called a \navy{straddle point} for $(\mu, \nu) \in \pP
\times \pP$ if $\mu([a,s])$ and $\nu([s,b])$ are both positive; that is, if $s$
lies above a positive fraction of $\mu$ and below a positive fraction of $\nu$.
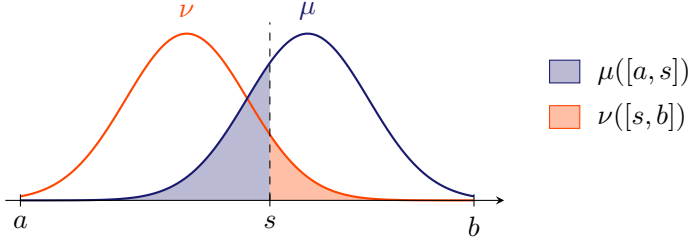
\begin{figure}[t]
    \centering
    \begin{tikzpicture}[>=stealth]
        \fill[MidnightBlue!30]
            plot[domain=0:3.3, samples=60] (\x, {2.2*exp(-(\x-3.8)^2/1.3)})
            -- (3.3,0) -- (0,0) -- cycle;
        \fill[OrangeRed!35]
            plot[domain=3.3:6, samples=60] (\x, {2.2*exp(-(\x-2.2)^2/1.3)})
            -- (6,0) -- (3.3,0) -- cycle;
        \draw[thick, OrangeRed]
            plot[domain=0:6, samples=120] (\x, {2.2*exp(-(\x-2.2)^2/1.3)});
        \draw[thick, MidnightBlue]
            plot[domain=0:6, samples=120] (\x, {2.2*exp(-(\x-3.8)^2/1.3)});
        \draw[->] (-0.2,0) -- (6.4,0);
        \foreach \p/\l in {0/a, 3.3/s, 6/b}
            \draw (\p,0.08) -- (\p,-0.08) node[below] {$\l$};
        \draw[dashed] (3.3,0) -- (3.3,2.35);
        \node[OrangeRed] at (2.2,2.5) {$\nu$};
        \node[MidnightBlue] at (3.8,2.5) {$\mu$};
        \fill[MidnightBlue!30, draw=MidnightBlue] (7.0,1.55) rectangle (7.4,1.85);
        \node[anchor=west] at (7.5,1.7) {$\mu([a,s])$};
        \fill[OrangeRed!35, draw=OrangeRed] (7.0,0.95) rectangle (7.4,1.25);
        \node[anchor=west] at (7.5,1.1) {$\nu([s,b])$};
    \end{tikzpicture}
    \caption{A straddle point $s$ for $(\mu, \nu)$ when $\Xsf \subset \RR$.}
    \label{f:straddle}
\end{figure}

A straddle point provides an easy test for partial dominance:

\begin{lemma}\label{l:straddle}
    If $s \in \Xsf$ is a straddle point for $(\mu, \nu) \in \pP \times \pP$,
    then $\mu$ is partially dominated by $\nu$ at level
    $\min \{ \mu([a,s]), \, \nu([s,b]) \}$.
\end{lemma}

\cref{f:straddle} illustrates the idea when $\Xsf$ is an interval.  The
shaded regions are the part of $\mu$ lying at or below $s$ and the part of
$\nu$ lying at or above $s$, and both have positive mass because $s$ is a
straddle point.  Scaling the larger of the two down until its mass equals that
of the smaller yields subpopulations $\mu' \leq \mu$ and $\nu' \leq \nu$ of
equal size $\epsilon = \min \{ \mu([a,s]), \, \nu([s,b]) \}$.  Moreover,
$\mu' \preqsd \nu'$, since all of the mass of $\mu'$ lies at or below $s$,
while all of the mass of $\nu'$ lies at or above it.  The same argument works
for a general partial order; a proof is given in \cref{s:proofs}.

\subsection{The Monotone Mixing Condition}\label{ss:mmc}

The \navy{monotone mixing condition} (MMC) of \cite{hopenhayn1992stochastic}
requires that there exist an $n \in \NN$ and an $s \in \Xsf$ with
\begin{equation}\label{eq:mmc}
    P^n(a, [s,b]) > 0
    \quad \text{and} \quad
    P^n(b, [a,s]) > 0 .
\end{equation}
In the terminology just introduced, \eqref{eq:mmc} says that $s$ is a
straddle point for $(P^n_b, P^n_a)$.  By \cref{l:straddle}, $P^n_b$ is then
partially dominated by $P^n_a$, so the MMC implies the WMMC, and
\cref{t:main} recovers the global stability result of
\cite{hopenhayn1992stochastic}.  The two conditions differ in how the reversal
of rank is generated.  The MMC requires a straddle point:  at some common
horizon, a positive fraction of the descendants of $b$ sits at or below $s$,
while a positive fraction of the descendants of $a$ sits at or above it.
The WMMC retains the reversal but discards the straddle point.  The converse
implication fails, and this matters for economic applications with multiple
state variables, as shown in \cref{s:app}.

\subsection{The Univariate Case}

In the univariate case, the WMMC and the MMC are equivalent: every
instance of partial dominance can be certified by a straddle point (see
\cref{p:real} in \cref{s:proofs}).  As a result, \cref{t:main} converts the MMC into an
exact characterization of global stability.

\begin{theorem}\label{t:mmc}
    Let $\Xsf = [a,b] \subset \RR$ with the usual order and let $P$ be an
    increasing stochastic kernel on $(\Xsf, \bB)$.  Then $P$ is globally stable
    if and only if $P$ satisfies the \textup{MMC}.
\end{theorem}

\section{Application: A Small Open Economy Model}
\label{s:app}

We now apply \cref{t:main} to a small open economy version of the
\cite{aiyagari1994uninsured} model. The distinguishing feature is that the
interest rate and the wage are tied together by a decreasing factor-price
relation. As a result, the two factor prices move in opposite directions, which
breaks the MMC of \cite{hopenhayn1992stochastic}, while the WMMC still holds.

\subsection{Household Sector and Factor Prices}

Time is discrete and there is a unit continuum of households. 
The economy is small open, so it takes the world interest rate as
given. The net world interest rate $r_t$ is {\sc iid} with a continuous
distribution supported on $[\underline r,\bar r]$. A competitive
firm produces a single good with a constant-returns-to-scale technology
$F(K,L)$, where $K$ and $L$ are the demands for capital and labor, respectively. 
Assume that $F$ is twice continuously differentiable, nonnegative, and
strictly increasing in each input. Capital
depreciates at rate $d\in(0,1)$, and $-d<\underline r<\bar r<\infty$.
Let $f(k)=F(k,1)$ denote the intensive production function, with
$k=K/L$. Then $f'>0$, and we assume in addition that $f''<0$ and that $f'$
satisfies the Inada conditions $f'(0)=\infty$ and $f'(\infty)=0$. Profit
maximization gives
\begin{equation*}
	r_t+d=f'(k_t),
	\qquad
	w_t=f(k_t)-k_t f'(k_t).
\end{equation*}
Because $f''<0$, and since $\underline r + d > 0$, the first equation has a
unique solution $k_t=\psi(r_t) \in (0, \infty)$. Substituting into the
second equation yields the factor-price relation
\begin{equation*}
	w_t=W(r_t),
	\qquad
	W(r)=f(\psi(r))-\psi(r)\bigl(r+d\bigr).
\end{equation*}
Moreover, $W'(r)=-\psi(r)<0$, so the wage is a strictly decreasing function
of the interest rate. We set $\underline w=W(\bar r)$ and $\bar w=W(\underline r)$.
Note that $\underline w>0$, because $f(0) \geq 0$ and $f''<0$ imply
$f(k)-kf'(k)>0$ for all $k>0$.
Thus $(r_t,w_t)$ always lies on a decreasing curve
\begin{equation*}
	D \coloneq \left\{(r, W(r)): r \in [\underline r, \bar r]\right\}.
\end{equation*}
Each household has a period utility function
$u:\RR_+\to \RR$ that is twice continuously differentiable,
strictly increasing, strictly concave, and satisfies the Inada
conditions $u'(0)=\infty$ and $u'(\infty)=0$. The household aims to solve
\begin{align*}
	& \text{maximize}  
	&& \EE_0\sum_{t=0}^{\infty}\beta^t u(c_t) \\
	&\text{subject to}   
	&& c_t+k_{t+1}=w_t z_t+(1+r_t)k_t,
	\qquad c_t, k_{t+1}\ge 0.
\end{align*}
Here $\beta\in (0,1)$ is the discount factor, $k_t$ is individual wealth held at 
the beginning of period $t$, and $z_t$ is idiosyncratic productivity. The 
productivity shock has a continuous distribution supported on 
$[\underline z,\bar z]$, where $0<\underline z<\bar z<\infty$, and the pair
$(z_t, r_t)$ is {\sc iid} over time.

\subsection{State Space and Optimal Policy}
The state vector of a household is $x=(k,z,r,w)$. We impose the impatience
condition $\beta(1+\bar r)<1$ and assume that asset holdings are
restricted to $[0,\bar k]$, where the lower bound is the borrowing constraint
and $\bar k$ is a finite cap on savings, which we take to be large enough that
the cap never binds.\footnote{The 
	cap serves only to compactify the state space. Under impatience and 
	standard additional restrictions on preferences, optimal wealth 
	accumulation is bounded, so $\bar k$ can be chosen sufficiently large 
	that optimal wealth always lies in the interior of $[0,\bar k]$; 
	see, e.g., \citet[Proposition~4]{acikgoz2018existence} and
	\citet[Propositions~7--9]{zhu2020existence}.}
The state space is then
\begin{equation}\label{eq:ex_state_space}
	\Xsf \coloneq
	[0,\bar k]
	\times
	[\underline z,\bar z]
	\times
	[\underline r,\bar r]
	\times
	[\underline w,\bar w],
\end{equation}
endowed with the pointwise partial order, which is closed.  Under this order,
$\Xsf$ has least element $a=(0,\underline z,\underline r,\underline w)$ and
greatest element $b=(\bar k,\bar z,\bar r,\bar w)$.  Note that the price pair
$(r_t,w_t)$ takes values in the curve $D$ rather than the full rectangle
$[\underline r,\bar r] \times [\underline w,\bar w]$, but the rectangle is the
right state space:  since $W$ is strictly decreasing, distinct points of $D$
are unordered, so $D$ has no least or greatest element.  Nor can $w$ simply be
dropped from the state:  with state $(k, z, r)$ and $w = W(r)$, the savings
policy $g(k, z, r, W(r))$ need not be monotone in $r$, since a higher interest rate
raises interest income but lowers the wage.  The redundant coordinate $w$ is
what makes the dynamics increasing, and, as we show below, the failure of the
MMC is the price.  The same issue arises whenever two exogenous state
variables are tied together by a decreasing relation, so the phenomenon is
not special to this model.
The Bellman equation is
\begin{equation}
	V(k,z,r,w)
	=
	\max_{0\le k'\le\bar k}
	\left\{
	u\bigl(wz+(1+r)k-k'\bigr)
	+
	\beta \EE V\bigl(k',z',r',W(r')\bigr)
	\right\},
	\label{eq:bellman}
\end{equation}
where the expectation is over the {\sc iid} shocks $(z',r')$, and consumption
is also required to be nonnegative.  The Bellman
operator is a contraction on $bc\Xsf$, so \eqref{eq:bellman} has a unique
solution $V \in bc\Xsf$.  Since $u$ is concave and the
constraint set is convex, $V$ is concave in $k$, and strict concavity of $u$
then implies that the maximizer is unique.  The resulting savings policy
$k'=g(k,z,r,w)$ takes values in $[0,\bar k]$, so $\Xsf$ is a compact invariant
state space.  By our choice of $\bar k$, we have $g(x) < \bar k$ for all $x \in
\Xsf$.

\subsection{Monotonicity, Mixing, and Global Stability}

Let $P$ be the stochastic kernel induced by the optimal policy. Given
$x=(k,z,r,w)$, the next-period state is
\begin{equation*}
	x'=\bigl(g(k,z,r,w),\,z',\,r',\,W(r')\bigr),
\end{equation*}
where $(z',r')$ is drawn from the exogenous continuous law.
Since $u'' < 0$, the objective in \eqref{eq:bellman} has increasing differences
in $k'$ and cash-on-hand $wz + (1+r)k$, while the constraint set is increasing
in the state.  Hence, by Topkis's theorem, the policy $g$ is increasing in all
four arguments. Since the next-period shocks have the
same law from every state, $P$ is an increasing stochastic kernel.

We now verify that the MMC fails.  Fix $n\ge 1$ and a candidate straddle point
$s=(s_k,s_z,s_r,s_w)\in \Xsf$, and consider the sets
\begin{equation*}
	A \coloneq D\cap\{(r,w):(r,w)\le(s_r,s_w)\}
	\quad \text{and} \quad
	B \coloneq D\cap\{(r,w):(r,w)\ge(s_r,s_w)\}.
\end{equation*}
If $(r,w)\in A$, then $r\le s_r$ and $W(r)\le s_w$; since $W$ is decreasing,
the first inequality gives $W(r)\ge W(s_r)$, so $A$ is nonempty only if
$s_w\ge W(s_r)$.  A symmetric argument shows that $B$ is nonempty only if
$s_w\le W(s_r)$.  Hence, if both sets are nonempty, then $s_w=W(s_r)$ and,
because $W$ is strictly decreasing, $A=B=\{(s_r,W(s_r))\}$.  At the same time,
the $(r,w)$-marginal of $P^n_x$ is the law of $(r,W(r))$,
which is supported on $D$ and atomless.  Since $P^n(b,[a,s])>0$ requires
$(r_n,w_n)\in A$ with positive probability, and $P^n(a,[s,b])>0$ requires
$(r_n,w_n)\in B$ with positive probability, these two conditions cannot both
hold.  Thus the MMC fails at every $n$ and every $s$.

We now show that the WMMC holds.  The first step is to show that the
borrowing constraint binds with positive probability.

\begin{lemma}\label{l:binding}
    There exist $n \in \NN$ and $p > 0$ such that
    $\PP\{k_n = 0 \given x_0 = b\} \geq p$.
\end{lemma}

\begin{proof}
    Let $c(x) \coloneq wz + (1+r)k - g(x)$ be the consumption policy.  It is
    continuous by the theorem of the maximum.  Moreover, $wz \geq \underline w
    \, \underline z > 0$ at every $x \in \Xsf$, so the Inada condition implies
    $c(x) > 0$; since $\Xsf$ is compact, $\underline c \coloneq \min_{x \in
    \Xsf} c(x) > 0$.
    Suppose, contrary to the claim, that $\PP\{k_n = 0 \given x_0 = b\} =
    0$ for every $n \in \NN$.  Then, starting from $b$, we have $0 < k_{t+1}
    < \bar k$ almost surely at every date, so the Euler equation holds with
    equality:  $u'(c_t) = \beta \, \EE_t [(1 + r_{t+1}) \, u'(c_{t+1})]$.
    Iterating gives
    \begin{equation*}
        u'(c_0)
        = \EE_0 \left[ \prod_{i=1}^n \beta (1 + r_i) \, u'(c_n) \right]
        \leq (\beta (1 + \bar r))^n \, u'(\underline c)
        \qquad (n \in \NN) .
    \end{equation*}
    Since $\beta(1 + \bar r) < 1$, letting $n \to \infty$ gives $u'(c_0)
    \leq 0$, a contradiction.
\end{proof}

In words, \cref{l:binding} states that a household that begins with maximal
wealth and productivity exhausts its assets within $n$ periods with positive
probability.  This is a standard scenario in Aiyagari-style models with
impatient households:  a sufficiently long run of unfavorable shocks drives
wealth down to the borrowing constraint. See, e.g.,
\citet[Lemma~F.2]{acikgoz2018existence}, \citet[Proposition~11]{zhu2020existence}, and \citet[Lemma~C.2]{ma2020income}.

Since the shocks $(z_t, r_t)$ are {\sc iid}, the time-$n$ state consists of
$k_n$, which depends only on $x_0$ and the shocks up to date $n-1$, and $(z_n,
r_n, w_n)$, which is independent of both.  Hence
\begin{equation}\label{eq:prod}
	P^n_x = \lambda^x_n \otimes \eta
	\qquad (x \in \Xsf),
\end{equation}
where $\otimes$ denotes the product of measures, $\lambda^x_n$ is the
distribution of $k_n$ given $x_0 = x$, and $\eta$ is the common law of $(z, r,
W(r))$.  Since $\delta_a \preqsd \delta_b$ and $P$
is increasing, we have $P^n_a \preqsd P^n_b$.  The set $\setntn{x \in \Xsf}{k
> 0}$ is increasing, so it receives no more mass under $P^n_a$ than under
$P^n_b$.  Taking complements,
\begin{equation*}
	\lambda^a_n(\{0\})
	\geq \lambda^b_n(\{0\})
	\geq p ,
\end{equation*}
where the second inequality holds by \cref{l:binding}.  Now let $\mu' = \nu' \coloneq p
\, (\delta_0 \otimes \eta)$, where $\delta_0$ is the unit mass at $k = 0$.  By
\eqref{eq:prod} and the previous display, $\mu' \leq P^n_b$ and $\nu' \leq
P^n_a$, while $\mu' \preqsd \nu'$ holds trivially because the two measures are
identical.  Hence $P^n_b$ is partially dominated by $P^n_a$ at level $p > 0$.
Thus $P$ satisfies the WMMC and, by
\cref{t:main}, is globally stable.  Moreover, the rate bound \eqref{eq:rate}
holds with this $n$ and $\epsilon = p$.

In summary, the descendants of both dynasties reach the borrowing constraint
simultaneously with positive probability, while facing the same distribution
over prices, and this overlap is enough for the WMMC:  the reversal it requires
is weak, so coalescence suffices.  The MMC fails
nonetheless, since it insists that the overlap be certified by a single
straddle point $s$, and no such point exists:  prices always lie on the
decreasing curve $D$, along which no two distinct points are comparable.

\section{Proofs}\label{s:proofs}

This section collects the proofs.  We begin with the straddle point lemma.

\begin{proof}[Proof of \cref{l:straddle}]
    Suppose that $s \in \Xsf$ is a straddle point for $(\mu, \nu)$.
    Let $p \coloneq \mu([a,s])$, $q \coloneq \nu([s,b])$, and $\epsilon \coloneq
    \min\{p, q\}$. Define $\mu'$ and $\nu'$ by
    \begin{equation}\label{eq:mpnp}
        \mu'(B) \coloneq \frac{\epsilon}{p} \, \mu(B \cap [a,s])
        \quad \text{and} \quad
        \nu'(B) \coloneq \frac{\epsilon}{q} \, \nu(B \cap [s,b])
        \qquad (B \in \bB).
    \end{equation}
    One easily confirms that $\mu' \leq \mu$, $\nu' \leq \nu$ and
    $\mu'(\Xsf) = \nu'(\Xsf) = \epsilon$.  To see that $\mu' \preqsd \nu'$,
    fix $I \in i\bB$.  If $s \in I$, then $[s,b] \subseteq I$ because $I$ is
    increasing, so $\nu'(I) = \epsilon$, while $\mu'(I) \leq \epsilon$.  If $s \notin I$,
    then $I \cap [a,s] = \emptyset$, since $y \in I$ and $y \preceq s$ would
    force $s \in I$, so $\mu'(I) = 0$.  In both cases
    $\mu'(I)  \leq \nu'(I)$.  The claim follows.
\end{proof}

The key tool in the proofs of \cref{t:main,t:mmc} is the ordered affinity of
\cite{kamihigashi2019unified}, which measures the extent to which one
distribution is partially dominated by another.  The \navy{ordered affinity}
of $(\mu, \nu) \in \pP \times \pP$ is the largest level at which partial
dominance holds:
\begin{equation}\label{eq:alf}
    \alpha(\mu, \nu)
    \coloneq \max
    \setntn{\epsilon \in [0,1]}
    {\text{$\mu$ is partially dominated by $\nu$ at level $\epsilon$}} .
\end{equation}
The maximum is attained \citep[Proposition~3.1]{kamihigashi2019unified}.
(At $\epsilon = 0$, \eqref{eq:pd} holds trivially with $\mu' = \nu' = 0$, so the
set in \eqref{eq:alf} is nonempty.)
In this language, the WMMC of \cref{s:main} states that
$\alpha(P^n_b, P^n_a) > 0$ for at least one $n \in \NN$, and, by attainment,
the sharpest version of the rate bound \eqref{eq:rate} at horizon $n$ sets
$\epsilon = \alpha(P^n_b, P^n_a)$.

We make frequent use of the dual expression
\begin{equation}\label{eq:dual}
    \alpha(\mu, \nu)
    = 1 - \sup_{I \in i\bB} \, \left( \mu(I) - \nu(I) \right)
    \qquad (\mu, \nu \in \pP) ,
\end{equation}
which is part of Theorem~3.1 of \cite{kamihigashi2019unified}.  One
immediate consequence is the monotonicity property
\begin{equation}\label{eq:amon}
    \hat \mu \preqsd \mu
    \;\; \text{and} \;\;
    \nu \preqsd \hat \nu
    \quad \implies \quad
    \alpha(\mu, \nu) \leq \alpha(\hat \mu, \hat \nu) ,
\end{equation}
valid for all $\mu, \nu, \hat \mu, \hat \nu \in \pP$.  Indeed, the
hypotheses of \eqref{eq:amon} give $\hat \mu(I) - \hat \nu(I) \leq \mu(I) -
\nu(I)$ for every $I \in i\bB$, and taking suprema in \eqref{eq:dual} yields
the conclusion.

Finally, we make use of the \navy{total ordered variation} distance between
$\mu$ and $\nu$, which is
\begin{equation}\label{eq:tov}
    \gamma(\mu, \nu) \coloneq 2 - \alpha(\mu, \nu) - \alpha(\nu, \mu) .
\end{equation}
That $\gamma$ is a metric on $\pP$ is Lemma~4.1 of
\cite{kamihigashi2019unified}.  Moreover, $\gamma$ is equivalent to $\kappa$.
Indeed, \eqref{eq:dual} gives
$\gamma(\mu, \nu) = \sup_{I \in i\bB} ( \mu(I) - \nu(I) ) +
\sup_{I \in i\bB} ( \nu(I) - \mu(I) )$, where both suprema are
nonnegative because $\emptyset \in i\bB$, while \eqref{eq:kol} says that
$\kappa(\mu, \nu)$ is twice the maximum of the same two suprema.  Hence
\begin{equation}\label{eq:gk}
    \gamma(\mu, \nu) \leq \kappa(\mu, \nu) \leq 2 \gamma(\mu, \nu)
    \qquad (\mu, \nu \in \pP) .
\end{equation}

We can now prove the main theorem.

\begin{proof}[Proof of \cref{t:main}]
    Suppose first that $P$ is globally stable, and let $\mu^*$ be its
    stationary distribution.  Since $P^t_b = \delta_b P^t$ and $P^t_a = \delta_a P^t$, global
    stability gives
    \begin{equation*}
        \kappa(P^t_b, P^t_a)
        \leq \kappa(P^t_b, \mu^*) + \kappa(\mu^*, P^t_a)
        \to 0
        \quad (t \to \infty) .
    \end{equation*}
    Since $\alpha \leq 1$, definition \eqref{eq:tov} gives $1 - \alpha(P^t_b,
    P^t_a) \leq \gamma(P^t_b, P^t_a)$, and $\gamma \leq \kappa$ by
    \eqref{eq:gk}.  Hence
    \begin{equation}\label{eq:atone}
        \alpha(P^t_b, P^t_a) \to 1
        \quad (t \to \infty) ,
    \end{equation}
    and in particular $P$ satisfies the WMMC.

    Suppose now that $P$ satisfies the WMMC, and fix $n \in \NN$ such that $P^n_b$ is
    partially dominated by $P^n_a$.  By \eqref{eq:alf}, this means that
    $\alpha_n \coloneq \alpha(P^n_b, P^n_a) > 0$.  We claim
    that
    \begin{equation}\label{eq:corner}
        \alpha(P^n_x, P^n_{x'}) \geq \alpha_n
        \qquad (x, x' \in \Xsf) .
    \end{equation}
    To see this, fix $x, x' \in \Xsf$.  Since $x \preceq b$ and $a \preceq x'$,
    we have $\delta_x \preqsd \delta_b$ and $\delta_a \preqsd \delta_{x'}$.  As
    $P^n$ is increasing, it follows that $\delta_x P^n \preqsd \delta_b P^n$ and
    $\delta_a P^n \preqsd \delta_{x'} P^n$, so \eqref{eq:amon} gives
    \begin{equation*}
        \alpha_n
        = \alpha(\delta_b P^n, \delta_a P^n)
        \leq \alpha(\delta_x P^n, \delta_{x'} P^n)
        = \alpha(P^n_x, P^n_{x'}) ,
    \end{equation*}
    which is \eqref{eq:corner}.  Since equality holds at $(x, x') = (b, a)$, the
    infimum of $\alpha(P^n_x, P^n_{x'})$ over $\Xsf \times \Xsf$ is exactly
    $\alpha_n > 0$.  This is the hypothesis of Corollary~5.1 of
    \cite{kamihigashi2019unified}, which therefore applies:  the kernel $P$ has a
    unique stationary distribution $\mu^* \in \pP$, and
    \begin{equation}\label{eq:grate}
        \gamma(\mu P^t, \mu^*)
        \leq (1 - \alpha_n)^{\lfloor t/n \rfloor} \gamma(\mu, \mu^*)
        \qquad (\mu \in \pP, \; t \geq 0) .
    \end{equation}
    In particular, by \eqref{eq:gk}, $\kappa(\mu P^t, \mu^*) \to 0$ for every
    $\mu \in \pP$, so $P$ is globally stable.

    Turning to the rate bound, let $n \in \NN$ and $\epsilon > 0$ be such
    that $P^n_b$ is partially dominated by $P^n_a$ at level $\epsilon$.  Then $\alpha_n \coloneq
    \alpha(P^n_b, P^n_a) \geq \epsilon > 0$ by \eqref{eq:alf}, so the argument
    leading to \eqref{eq:grate} applies at this $n$.  Combining
    \eqref{eq:grate} with $\alpha_n \geq \epsilon$ and \eqref{eq:gk} gives
    \begin{equation*}
        \kappa(\mu P^t, \mu^*)
        \leq 2 \gamma(\mu P^t, \mu^*)
        \leq 2 (1 - \epsilon)^{\lfloor t/n \rfloor} \gamma(\mu, \mu^*)
        \leq 4 (1 - \epsilon)^{\lfloor t/n \rfloor}
    \end{equation*}
    for all $\mu \in \pP$ and $t \geq 0$, where the last step uses
    $\gamma(\mu, \mu^*) \leq \kappa(\mu, \mu^*) \leq 2$.  This is
    \eqref{eq:rate}.
\end{proof}

Now we turn to the univariate case.
The proof of \cref{t:mmc} rests on the following result. 

\begin{proposition}\label{p:real}
    If $\Xsf = [a,b] \subset \RR$ with the usual order, then, given $\mu, \nu
    \in \pP$, the distribution $\mu$ is partially dominated by $\nu$ if and
    only if there is an $s \in \Xsf$ with $\mu([a,s]) > 0$ and
    $\nu([s,b]) > 0$.
\end{proposition}

\begin{proof}
    Sufficiency is immediate from \cref{l:straddle}.  For necessity we argue by
    contraposition.  Assume that, for every $s \in [a,b]$, either $\mu([a,s]) =
    0$ or $\nu([s,b]) = 0$.  Set
    \begin{equation*}
        \tau \coloneq \sup \left( \{a\} \cup
        \setntn{r \in [a,b]}{\mu([a,r]) = 0} \right) .
    \end{equation*}
    Since $r \mapsto \mu([a,r])$ is nondecreasing, the set in question is an
    interval containing its smaller points, so $\mu([a,r]) = 0$ for every $r < \tau$
    and $\mu([a,r]) > 0$ for every $r > \tau$.  In particular $\mu([a,\tau)) = 0$.
    Define $I$ as $[\tau, b]$ if $\mu(\{\tau\}) > 0$ and $(\tau, b]$ otherwise.
    In either case $I$ is an increasing Borel subset of $\Xsf$ with $\mu(I) = 1$.
    We claim that $\nu(I) = 0$.  Given the claim, $\mu(I) - \nu(I) = 1$, so the
    supremum in \eqref{eq:dual} equals one and $\alpha(\mu, \nu) = 0$.  By
    \eqref{eq:alf}, $\mu$ is then not partially dominated by $\nu$ at any
    positive level, and the proof is complete.

    Suppose first that $\mu(\{\tau\}) > 0$.  Then $\mu([a,\tau]) > 0$, so $\nu([\tau,b])
    = 0$ by hypothesis, and $I = [\tau,b]$. Suppose instead that $\mu(\{\tau\}) = 0$.
    Then $\mu([a,\tau]) = 0$, and since $\mu([a,b]) = 1$ this forces $\tau < b$.  For
    every $s \in (\tau, b]$ we have $\mu([a,s]) > 0$ and hence $\nu([s,b]) = 0$.
    Letting $s$ decrease to $\tau$ along $(\tau,b]$ gives $\nu(I) = \nu((\tau,b]) = 0$.
\end{proof}

\begin{proof}[Proof of \cref{t:mmc}]
    If $P$ satisfies the MMC, then it satisfies the WMMC by \cref{l:straddle}
    and is therefore globally stable by \cref{t:main}.  Conversely, suppose that $P$ is globally stable.  By
    \cref{t:main} there is an $n \in \NN$ such that $P^n_b$ is partially
    dominated by $P^n_a$, and \cref{p:real}, applied with $\mu =
    P^n_b$ and $\nu = P^n_a$, supplies an $s \in \Xsf$ with $P^n(b, [a,s]) > 0$
    and $P^n(a, [s,b]) > 0$.  This is \eqref{eq:mmc}.
\end{proof}

\bibliographystyle{ecta}

\bibliography{localbib}

\end{document}